\documentclass[12pt,reqno]{amsart}
\usepackage{mathtools,xcolor}

\mathtoolsset{showonlyrefs}
 \hoffset=-56pt
\usepackage{eucal}
\usepackage{enumerate}
\usepackage{amssymb,comment}
\usepackage{cite}
\usepackage{bm} \numberwithin{equation}{section}

\usepackage{mathtools}

\newcommand{\ko}{\mathcal K_0^n}

\newcommand{\ke}{\mathcal K_e^n}
\newcommand{\rn}{\mathbb R^n}

\newcommand{\sn}{ {S^{n-1}}}

\newtheorem{lemma}{Lemma}[section]
\newtheorem{theorem}[lemma]{Theorem}

\newtheorem{example}[lemma]{Example}
\newtheorem{defi}[lemma]{Definition}
\newtheorem{coro}[lemma]{Corollary}

\newtheorem{remark}[lemma]{Remark}
\title{The Minkowski problem in Gaussian probability space}
\author[Y. Huang]{Yong Huang}
\address{Institute of Mathematics, Hunan University, 2 Lushan S Road, 410082, Changsha, China}
\email{huangyong@hnu.edu.cn}
\author[D. Xi]{Dongmeng Xi}
\address{Department of Mathematics, Shanghai University, 266 Jufengyuan Rd, 200444, Shanghai, China; Courant Institute, New York University, 251 Mercer St, 10012, New York, USA}
\email{dongmeng.xi@live.com}
\author[Y. Zhao]{Yiming Zhao}
\address{Department of Mathematics,  Massachusetts Institute of Technology, 77 Massachusetts Ave, 02139, Cambridge, USA}
\email[corresponding author]{yimingzh@mit.edu}

\thanks{Research of Huang is supported by the National Science Fund for Distinguished Young Scholars
(11625103), Tian Yuan Special Foundation (11926317) and Hunan Science and Technology Planning Project
(2019RS3016). Research of Xi is supported by National Natural Science Foundation of China (12071277), and STSCM program (20JC1412600). Research of Zhao is supported, in part, by U.S. National Science Foundation Grant DMS-2002778.}

\begin{document}
\begin{abstract}
	The Minkowski problem in Gaussian probability space is studied in this paper. In addition to providing an existence result on a Gaussian-volume-normalized version of this problem, the main goal of the current work is to provide uniqueness and existence results on the Gaussian Minkowski problem (with no normalization required). 
\end{abstract}
\maketitle
\section{Introduction}

The seminal work \cite{minkowski2} by Minkowski in 1903 can be viewed as the starting point of the now vibrant Brunn-Minkowski theory in convex geometry. Minkowski's work relied on a remarkable result by Steiner in 1840 that combines the power of the usual operations in the Euclidean space with the usual Lebesgue measure in the space. The so-called \emph{Steiner formula} states that in $\rn$, the Lebesgue measure of $K+t\cdot B=\{x+ty: x\in K, y\in B\}$ where $K$ is a convex body (compact convex set), $B$ is the Euclidean unit ball and $t>0$ is a polynomial of degree $n$ in $t$ with coefficients carrying essential geometric information regarding $K$. These coefficients are known as \emph{quermassintegrals} and include volume, surface area, mean width and many more geometric invariants. Inequalities and their equality conditions involving these invariants can then be used to identify geometric shapes---perhaps the most well-known one is the isoperimetric inequality that identifies balls. These invariants, when being ``differentiated'', generate geometric measures that arguably carry more geometric information and at times all information as they can be used to uniquely recover the geometric shape. The celebrated Minkowski problem is one such example (perhaps the most well-known one). Minkowski asked if a given Borel measure $\mu$ on $\sn$ can be used to reconstruct a convex body whose \emph{surface area measure} is precisely the given measure $\mu$ and if the reconstruction is unique. Here, the surface area measure of $K$, denoted by $S_K$, is uniquely determined by the Aleksandrov's variational formula   
\begin{equation}
\label{eq local 9002}
	\lim_{t\rightarrow 0} \frac{\mathcal{H}^{n}(K+tL)-\mathcal{H}^n(K)}{t} = \int_{\sn} h_L(v)dS_K(v)
\end{equation}
where $h_L:\sn\rightarrow \mathbb{R}$ is the support function of $L$ (see \eqref{eq local 0004}). The influence of the Minkowski problem is widespread. In differential geometry, this is the problem of the prescription of Gauss curvature; in nonlinear PDE, it has the appearance of Monge-Amp\`{e}re equation. For the last three decades, there have been many Minkowski-type problems, each involving a certain geometric measures generated by ``differentiating'' an invariant in Steiner's formula in a way such as in \eqref{eq local 9002}. These Minkowski problems can be understood as the problems of reconstructing convex bodies in manners specified by the geometric measures in question and each of them, when asked in the smooth category, reduces to a certain fully nonlinear elliptic PDE of varying natures. Some of the most prominent Minkowski-type problems include the $L_p$ Minkowski problem (see \cite{MR1231704, MR2132298,MR2254308}), the logarithmic Minkowski problem (see \cite{BLYZ}), and the dual Minkowski problem (see \cite{HLYZ}). We shall provide a short review of these problems shortly. 

Perhaps of equal significance as the Lebesgue measure in $\rn$ is the Gaussian probability measure $\gamma_n$ given by
\begin{equation*}
	\gamma_n(E)=\frac{1}{(\sqrt{2\pi})^n} \int_{E}e^{-\frac{|x|^2}{2}}dx.
\end{equation*} 
Unlike Lebesgue measure, Gaussian probability measure is neither translation invariant nor homogeneous. Moreover, the density decays exponentially fast as $|x|\rightarrow \infty$. The ``surface area measure'' in the Gaussian probability space is known as the \emph{Gaussian surface area measure}, which was studied in, for example, Ball \cite{MR1243336} and Nazarov \cite{MR2083397}. In this paper, we will retrace the steps of Minkowski, Aleksandrov among many others and study the corresponding Minkowski problem in Gaussian probability space. As we will see shortly, the missing features such as translation invariance and homogeneity, along with exponential decay, causes the behavior of the \emph{Gaussian Minkowski problem} to be quite mysterious and differs significantly from that of the Minkowski problem. 

The following variational formula allows us to ``differentiate'' the Gaussian volume $\gamma_n(\cdot)$ on the set of convex bodies:
\begin{equation*}
	\lim_{t\rightarrow 0} \frac{\gamma_n(K+tL)-\gamma_n(K)}{t} = \int_{\sn } h_L dS_{\gamma_n, K},
\end{equation*}
for any convex bodies $K$ and $L$ containing the origin in their interiors. The proof will be given in Theorem \ref{thm variational formula}. The uniquely determined Borel measure $S_{\gamma_n, K}$ is defined, in an equivalent way, in \eqref{eq local 7001} and will be referred to as the \emph{Gaussian surface area measure} of $K$ for its corresponding role in Gauss probability space when compared to surface area measure in the Lebesgue measure space. When $K$ is sufficiently smooth, its Gaussian surface area measure is absolutely continuous with respect to spherical Lebesgue measure:
\begin{equation}
\label{eq local 9040}
	dS_{\gamma_n, K}(v) = \frac{1}{(\sqrt{2\pi})^n} e^{-\frac{|\nabla h_K|^2+h_K^2}{2}}\det(\nabla^2 h_K + hI),
\end{equation}
where $h_K:\sn\rightarrow \mathbb{R}$ is the support function of $K$ and $\nabla$, $\nabla^2$ are gradient and Hessian operators on $\sn$ with respect to the standard metric. 

It is therefore natural to wonder what measures can be used to reconstruct convex bodies in Gaussian probability space based on their Gaussian surface area measure and whether Gaussian surface area measure uniquely identifies the body. 

\textbf{The Gaussian Minkowski problem.} Given a finite Borel measure $\mu$, what are the necessary and sufficient conditions on $\mu$ so that there exists a convex body $K$ with $o\in \text{int}\,K$ such that
\begin{equation}
\label{eq local 00005}
	\mu  = S_{\gamma_n, K}?
\end{equation}
If $K$ exists, to what extent is it unique?

Because of \eqref{eq local 9040}, when the given measure $\mu$ has a density $d\mu = fdv$, the Gaussian Minkowski problem reduces to solving the following Monge-Amp\`{e}re type equation on $\sn$,
\begin{equation*}
	\frac{1}{(\sqrt{2\pi})^n} e^{-\frac{|\nabla h|^2+h^2}{2}}\det (\nabla^2 h+hI) = f.
\end{equation*}

By the works of Ball \cite{MR1243336} and Nazarov \cite{MR2083397}, it is simple to notice that the allowable $\mu$ in the Gaussian Minkowski problem cannot have an arbitrarily big total mass. In fact, the Gaussian surface area of any convex set in $\mathbb{R}^n$ is up to a constant  bounded from above by $n^\frac{1}{4}$.   

We will briefly discuss several other features that distinguish the Gaussian Minkowski problem from the Minkowski problem in Lebesgue measure space, which are what makes the problem more interesting and simultaneously more challenging. 

To start, notice that when $K$ is a centered ball of radius $r$, according to \eqref{eq local 9040}, the density of its Gaussian surface area measure is given by $f_r\equiv \frac{1}{(\sqrt{2\pi})^n}e^{-r^2/2}r^{n-1}$. Notice that $e^{-r^2/2}r^{n-1}\rightarrow 0$ both when $r$ approaches $0$ and $\infty$. Thus, in full generality, even when $\mu = cdv$ for some constant $c>0$, the solutions to the Gaussian Minkowski problem are not unique. This is a result of the fact that the Gaussian probability space ``thins'' out exponentially as you travel away from the origin and therefore both larger and smaller convex bodies in $\rn$ can have relatively small Gaussian surface area. However, as we will show in Section 4, when restricted to convex bodies with larger than $1/2$ Gaussian volume, uniqueness part of the Gaussian Minkowski problem can be established. 
\begin{theorem}
\label{thm intro uniqueness}
Suppose $K, L$ are two convex bodies in $\rn$ that contain the origin in their interiors and $K, L$ both solve the Gaussian Minkowski problem; i.e.,
	\begin{equation*}
		S_{\gamma_n, K} = S_{\gamma_n, L}= \mu.
	\end{equation*}
	If $\gamma_n(K),\gamma_n(L)\geq 1/2$, then $K=L$.
\end{theorem}

Our uniqueness result utilizes the Ehrhard's inequality with its equality condition, along with several of its consequences. Ehrhard's inequality is an isoperimetric inequality in the Gaussian probability space and implies that half-spaces, among all other sets of the same Gaussian volume, attain the least Gaussian surface area. As mentioned before, the Gaussian probability measure does not enjoy any homogeneity. As a result, there are many isoperimetric inequalities in the Gaussian probability space. Of particular interest is the dimensional Brunn-Minkowski inequality (for $o$-symmetric convex bodies) conjectured by Gardner-Zvavitch \cite{MR2657682}, with important contribution by Kolesnikov-Livshyts \cite{KL1} followed by a recent confirmation by Eskenazis-Moschidis \cite{alex2020dimensional}. Gardner-Zvavitch \cite{MR2657682} observed that this inequality neither implies nor is implied by Ehrhard's inequality. The dimensional Brunn-Minkowski inequality is also linked with the conjectured log-Brunn-Minkowski inequality (planar case established in \cite{MR2964630})---an inequality in the Lebesgue measure space but with different addition---following a result by Livshyts-Marsiglietti-Nayar-Zvavitch \cite{MR3710641}, which was very recently extended in \cite{hosle2020l_p}. We also would like to mention the work of Borell \cite{MR399402}.

Since Gaussian probability measure is not translation-invariant, the position of a convex body with respect to the origin is of critical importance. In terms of the existence part of the Gaussian Minkowski problem, we will restrict ourselves to the $o$-symmetric case; that is, when the given measure $\mu$ is an even measure and the solution set is the set of all $o$-symmetric convex bodies. As Example \ref{example 1} in Appendix shows, even in this restricted case, the characterization of permissible measures $\mu$ are quite complicated. The situation is made worse by the lack of homogeneity in the Gaussian probability space (and therefore the Gaussian volume and surface area measure). Minkowski-type problems in which one deals with non-homogeneous geometric measures are typically known as of Orlicz type, which has their origin in the work \cite{MR2652213} by Haberl-Lutwak-Yang-Zhang for the Orlicz Minkowski problem that generalizes both the classical Minkowski problem and the $L_p$ Minkowski problem. See also \cite{MR3209355,MR3255458,MR3263511, MR3663326,MR3735745,MR3897433} for additional results and contributions in the Orlicz extension of the classical Brunn-Minkowski theory. Very recently, Gardner-Hug-Weil-Xing-Ye extended the recently posed dual Minkowski problem to its Orlicz counterpart \cite{MR3882970, MR4040624}. Inspired by their work, particularly the work \cite{MR2652213} by Haberl-Lutwak-Yang-Zhang, we obtain the following normalized solution to the even Gaussian Minkowski problem.
\begin{theorem}
\label{thm intro normalized problem}
	Suppose $\mu$ is an finite even Borel measure not concentrated in any closed hemisphere. Then for each $0<\alpha<\frac{1}{n}$, there exists an $o$-symmetric convex body $K$ such that 
	\begin{equation*}
		\mu = cS_{\gamma_n, K},
	\end{equation*}
	where $$c=\frac{1}{\gamma_n(K)^{1-\alpha}}.$$
\end{theorem}

The proof of Theorem \ref{thm intro normalized problem} is contained in Section \ref{subsection normalized solution}, which is of variational nature. If the Gaussian surface area measure $S_{\gamma_n, K}$ was homogeneous in $K$, one would then be able to get rid of the constant $c$. This, however, is \emph{far} from a simple procedure and as a matter of fact, to the best knowledge of the authors', every solution to Orlicz-Minkowski-type problem in the works mentioned above are all normalized solutions (meaning that there exists a constant $c$ in the solutions). It is unclear whether the reconstruction process is unique by allowing such a constant $c$ in the solution. Therefore, it is much desired to obtain a solution to the non-normalized version of the Gaussian Minkowski problem as stated in \eqref{eq local 00005} where uniqueness to a certain extent is guaranteed by one of our results (Theorem \ref{thm intro uniqueness}). One of our main results in the current paper is a progress in this direction by obtaining an existence result of the Gaussian Minkowski problem (non-normalized, restricted to $o$-symmetric case) via a degree theory approach. In particular, we will show
\begin{theorem}[Existence of smooth solutions]
\label{thm local 1}
	Let $0<\alpha<1$ and $f\in C^{2,\alpha}(\sn)$ be a positive even function with $|f|_{L_1}<\frac{1}{\sqrt{2\pi}}$. Then there exists a unique $C^{4,\alpha}$ $o$-symmetric $K$ with $\gamma_n(K)>1/2$ such that
	\begin{equation}
	\label{eq local 7030}
		\frac{1}{(\sqrt{2\pi})^n}e^{-\frac{|\nabla h_K|^2+h_K^2}{2}}\det (\nabla^2 h_K +h_KI) = f.
	\end{equation}
\end{theorem}
An approximation argument is then used to obtain an existence result (weak solution) for the Gaussian Minkowski problem.
\begin{theorem}
\label{thm local 2}
		Let $\mu$ be an even measure on $\sn$ that is not concentrated in any subspace and $|\mu|< \frac{1}{\sqrt{2\pi}}$. Then there exists a unique origin-symmetric $K$ with $\gamma_n(K)> 1/2$ such that
		\begin{equation*}
			S_{\gamma_n,K} = \mu.
		\end{equation*}
\end{theorem}

Notice that Theorem \ref{thm local 2} \emph{trivially} implies that if $\mu$ is an even measure on $\sn$ that is not concentrated in any subspace, there are infinitely many pairs of $c$ and $K$ such that $\mu = cS_{\gamma_n, K}$.

The author would like to point out that \emph{a differently normalized version} of the Gaussian Minkowski problem is a special case of the general dual Orlicz-Minkowski problem considered in \cite{MR3882970, MR4040624}. However, it is important to note that none of the main theorems in the current paper overlap with those presented there. One should also note that the Minkowski problem in measurable spaces whose densities possess certain homogeneity and concavity has been previously considered in Livshyts \cite{MR4008520}. 

Before ending this section, a short review of the aforementioned Minkowski-type problems in $\mathbb{R}^n$ with Lebesgue measure will be provided given their relevance to the current work and their importance in convex geometry. However, the more eager readers should feel free to skip it and jump to Section 2. 

Nine decades after Minkowski's seminal work \cite{minkowski2}, in the early 1990s, Lutwak \cite{MR1231704,MR1378681} laid the foundation to the now fruitful $L_p$ Brunn-Minkowski theory. Due to limit of space, we mention only a selection of beautiful results in this area, \cite{MR2530600,MR1987375, MR3415694,MR2964630,MR3148545, MR2652209, MR2680490, MR2262841, MR2880241,MR1901250,MR2019226,MR2729006, MR1863023} and refer the interested readers to Schneider's book \cite{schneider2014} for more details. In \cite{MR1231704}, Lutwak introduced the $L_p$ surface area measure which is the counterpart of the classical surface area measure in the $L_p$ theory and posed the corresponding $L_p$ Minkowski problem. When $p=1$, the $L_p$ Minkowski problem is precisely the classical Minkowski problem Lutwak himself solved the problem when $p> 1$ in the $o$-symmetric case whereas the more general case (non-symmetric) was settled by \cite{MR2132298,MR2254308}. The $L_p$ Minkowski problem when $p<1$ is much more complicated and contains \emph{challenging} problems such as the logarithmic Minkowski problem ($p=0$) and the centro-affine Minkowski problem ($p=-n$).

The logarithmic Minkowski problem characterizes cone volume measure which has been the central topic in a number of recent works. When the given data is even, the existence of solutions to the logarithmic Minkowski problem was completely solved in B\"{o}r\"{o}czky-Lutwak-Yang-Zhang \cite{BLYZ}. In the general case (non-even case), important contributions were made by Zhu \cite{MR3228445}, and later by B\"{o}r\"{o}czky-Heged\H{u}s-Zhu \cite{Boroczky20062015}. The logarithmic Minkowski problem has strong connections with isotropic measures (B{\"o}r{\"o}czky-Lutwak-Yang-Zhang \cite{MR3316972}) and curvature flows (Andrews \cite{MR1714339,MR1949167}).

The centro-affine Minkowski problem characterizes the centro-affine surface area measure whose density in the smooth case is the centro-affine Gauss curvature. The characterization problem, in this case, is the centro-affine Minkowski problem posed in Chou-Wang \cite{MR2254308}. See also Jian-Lu-Zhu \cite{MR3479715}, Lu-Wang \cite{MR2997361}, Zhu \cite{MR3356071}, \emph{etc.}, on this problem.

The readers are also referred to \cite{MR3872853, MR3680945} for some recent development of the $L_p$ Minkowski problem when $p<1$. 

The Minkowski problem and the $L_p$ Minkowski problem are within the framework introduced by Brunn and Minkowski. A parallel theory, which is known as the dual Brunn-Minkowski theory, was introduced by Lutwak (see Schneider \cite{schneider2014}) in the 1970s. The dual Brunn-Minkowski theory has been most effective in answering questions related to intersections. One major triumph of the dual Brunn-Minkowski theory is tackling the famous Busemann-Petty problem, see Gardner \cite{MR1298719}, Gardner-Koldobsky-Schlumprecht \cite{MR1689343}, Koldobsky \cite{MR1637955,MR1985195,MR2836117}, Lutwak \cite{MR963487}, and Zhang \cite{MR1689339}. The dual theory makes extensive use of techniques from harmonic analysis. Recently, the dual Brunn-Minkowski theory took a huge step forward when Huang-Lutwak-Yang-Zhang \cite{HLYZ} discovered the family of fundamental geometric measures---called dual curvature measures---in the dual theory. The dual Minkowski problem is the problem of prescribing dual curvature measures. The dual Minkowski problem introduces intrinsic PDEs---something long missing---to the dual Brunn-Minkowski theory. The dual Minkowski problem, while still largely open, has been  solved in the $o$-symmetric case when the associated index $q$ satisfies $q\in [0,n]$, see, for example \cite{HLYZ, MR3825606, MR3605843, MR3880233, MR4008522,MR3725875,MR3818073,MR4055992}.

The current work is along the lines of the classical problem raised by Minkowski, but now considered in the Gaussian probability space rather than the Lebesgue measure space. The lack of translation-invariance and homogeneity in the Gaussian probability space creates many challenges not encountered in the classical Minkowski problem.

\section{Preliminaries}
Some basics, as well as notations, regarding convex bodies will be provided in this section. For a general reference on the theory of convex bodies, the readers are referred to the book \cite{schneider2014} by Schneider.

Let $\rn$ be the $n$-dimensional Euclidean space. The unit sphere in $\rn$ is denoted by $\sn$. We will write $C(S^{n-1})$ for the space of continuous functions on $S^{n-1}$. We will use the subscript $e$ for even functions and the superscript $+$ for positive function so that $C_e^+(\sn)$ is used to denote the set of all even positive functions on $\sn$. For a Borel measure $\mu$ in a measure space, we will use $|\mu|$ for its total measure.

A convex body in $\rn$ is a compact convex set with nonempty interior. The boundary of $K$ is written as $\partial K$. Denote by $\ko$ the class of convex bodies that contain the origin in their interiors in $\rn$ and by $\ke$ the class of origin-symmetric convex bodies in $\rn$.

Let $K$ be a compact convex subset of $\rn$. The support function $h_K$ of $K$ is defined by
\begin{equation}
\label{eq local 0004}
	h_K(y) = \max\{x\cdot y : x\in K\}, \quad y\in\rn.
\end{equation}
The support function $h_K$ is a continuous function homogeneous of degree 1. Suppose $K$ contains the origin in its interior. The radial function $\rho_K$ is defined by
\[ \rho_K(x) = \max\{\lambda : \lambda x \in K\}, \quad x\in \rn\setminus \{0\}. \]
The radial function $\rho_K$ is a continuous function homogeneous of degree $-1$. It is not hard to see that $\rho_K(u)u \in
\partial K$ for all $u\in S^{n-1}$.

For each $f\in C^+(\sn)$, the Wulff shape $\bm{[}f\bm{]}$ generated by $f$ is the convex body defined by
\begin{equation*}
\bm{[}f\bm{]}= \{x\in \rn: x\cdot v \leq f(v), \text{ for all }v \in \sn\}.
\end{equation*}
It is apparent that $h_{\bm{[}f\bm{]}}\leq f$ and $\bm{[}h_K\bm{]} = K$ for each $K\in \ko$.

Suppose $K_i$ is a sequence of convex bodies in $\rn$. We say $K_i$ converges to a compact convex subset $K\subset \rn$ in Hausdorff metric if
\begin{equation}
\label{eq convergence convex bodies}
\max\{|h_{K_i}(v)-h_K(v)|:v\in S^{n-1}\}\rightarrow 0,
\end{equation}
as $i\rightarrow \infty$. If $K$ contains the origin in its interior, equation \eqref{eq convergence convex bodies} implies
\begin{equation*}
\max\{|\rho_{K_i}(u)-\rho_K(u)|:u\in S^{n-1}\}\rightarrow 0,
\end{equation*}
as $i\rightarrow \infty$.

For a compact convex subset $K$ in $\rn$ and $v \in \sn$, the supporting hyperplane $H(K,v)$ of $K$ at $v$ is given by
\begin{equation*}
H(K,v)=\{x\in K: x\cdot v = h_K(v)\}.
\end{equation*}
By its definition, the supporting hyperplane $H(K,v)$ is non-empty and contains only boundary points of $K$. For $x\in H(K,v)$, we say $v$ is an outer unit normal of $K$ at $x\in \partial K$. 

Since $K$ is convex, for $\mathcal{H}^{n-1}$ almost all $x\in \partial K$, the outer unit normal of $K$ at $x$ is unique. In this case, we use $\nu_K$ to denote the Gauss map that takes $x\in \partial K$ to its unique outer unit normal. Therefore, the map $\nu_K$ is almost everywhere defined on $\partial K$. We use $\nu_K^{-1}$ to denote the inverse Gauss map. Since $K$ is not assumed to be strictly convex, the map $\nu_K^{-1}$ is set-valued map and for each set $\eta\subset \sn$, we have
\begin{equation*}
	\nu_K^{-1}(\eta) = \{x\in \partial K: \text{there exists } v\in \eta \text{ such that } v \text{ is an outer unit normal at }x\}.
\end{equation*} 

Occasionally, for simplicity, we will sometimes use the following renormalization of the Gauss and inverse Gauss map. 

For those $u\in \sn$ such that $\nu_K$ is well-defined at $\rho_K(u)u\in \partial K$, we write $\alpha_K(u)$ for $ \nu_K(\rho_K(u)u)$. 

Let $\eta \subset \sn$ be a Borel set. The reverse radial Gauss image of $K$, denoted by $\alpha_K^*(\eta)$, is defined to be the set of all radial directions such that the corresponding boundary points have at least one outer unit normal in $\eta$, i.e.,
\begin{equation*}
\alpha_K^*(\eta) = \{u\in \sn: v\cdot u\rho_K(u) = h_K(v) \text{ for some } v \in \eta\}.
\end{equation*}
When $\eta = \{v\}$ is a singleton, we usually write $\alpha_K^*(v)$ instead of the more cumbersome notation $\alpha_K^*(\{v\})$. It follows from Theorem 2.2.11 in \cite{schneider2014} that for $\mathcal{H}^{n-1}$ almost all $v\in \sn$, the set $\alpha_K^*(\eta)$ contains only a singleton. Thus, we will sometimes treat $\alpha_K^*$ as an almost everywhere defined map on $\sn$ when no confusion arises. 

We recall that by Lemma 2.2 in \cite{HLYZ} that if $K_i$ converges to $K_0\in \mathcal{K}_o^n$ in Hausdorff metric, then $\alpha_{K_i}$ converges to $\alpha_{K_0}$ almost everywhere on $\sn$ with respect spherical Lebesgue measure.

\section{Gaussian surface area measure and the Gaussian Minkowski problem}

The purpose of this section is to introduce Gaussian surface area measure and the Gaussian Minkowski problem, and prove some basic properties as well as basic statements made in the Introduction but were not proved there. The authors would like to point out that these definitions and properties have already appeared in previous literatures (for example, \cite{MR3882970, MR4040624}) and are only included in this paper for the sake of completeness. With the exception of Theorem \ref{thm solution to the normalized problem}, we take no credit for the other results presented in this section.

We define the following Borel measure on $\sn$ and refer to it as \emph{Gaussian surface area measure}.

\begin{defi}
	Let $K\in \mathcal{K}_o^n$. The Gaussian surface area measure of $K$, denoted by $S_{\gamma_n, K}$, is a Borel measure on $\sn$ given by 
	\begin{equation}
	\label{eq local 7001}
			S_{\gamma_n, K}(\eta) = \frac{1}{(\sqrt{2\pi})^n}\int_{\nu_K^{-1}(\eta)} e^{-\frac{|x|^2}{2}}d\mathcal{H}^{n-1}(x),
	\end{equation}
	for each Borel measurable $\eta\subset \sn$.
\end{defi}

We will need the following lemma.
\begin{lemma}
\label{lemma local 10}
	Let $K\in \mathcal{K}_o^n$ and $f\in C({\sn})$. Suppose $\delta>0$ is sufficiently small so that for each $t\in (-\delta,\delta)$, we have
	\begin{equation*}
		h_t=h_{K}+tf>0.
	\end{equation*} 
	Then,
	\begin{equation*}
		\lim_{t\rightarrow 0} \frac{\rho_{[h_t]}(u)-\rho_K(u)}{t}=\frac{f(\alpha_K(u))}{h_K({\alpha_K}(u))}\rho_K(u)
	\end{equation*}
	for almost all $u\in\sn$ with respect to spherical Lebesgue measure. Moreover, there exists $M>0$, such that
	\begin{equation*}
		|\rho_{[h_t]}(u)-\rho_K(u)|<M|t|,
	\end{equation*}
	for all $u\in \sn$ and $t\in (-\delta,\delta)$. 
\end{lemma}
\begin{proof}
	The desired result follows immediately from Lemmas 2.8 and 4.1 in \cite{HLYZ} and also the fact that
	\begin{equation*}
		\log h_t = \log h_K + t\frac{f}{h_K} + o(t).
	\end{equation*}
\end{proof}

The following variational formula gives rise to the corresponding surface area measure in the Gaussian probability space and therefore justifies the name \emph{Gaussian surface area measure}. The proof is an adaptation of the variational formula obtained in \cite{HLYZ}.
\begin{theorem}
\label{thm variational formula}
	Let $K\in \mathcal{K}_o^n$ and $f\in C(S^{n-1})$. Then,
	\begin{equation}
	\label{eq local 7002}
		\lim_{t\rightarrow 0} \frac{\gamma_n([h_K+tf])-\gamma_n(K)}{t} = \int_{\sn } f dS_{\gamma_n, K}.
	\end{equation}
\end{theorem}
\begin{proof}
	Write $h_t= h_K+tf$. Using polar coordinates, we have
	\begin{equation*}
		\gamma_n([h_t]) = \frac{1}{(\sqrt{2\pi})^n}\int_{\sn}\int_{0}^{\rho_{[h_t]}(u)}e^{-\frac{r^2}{2}} r^{n-1}drdu. 
	\end{equation*}
	
	Since $K\in \mathcal{K}_o^n$ and $f\in C(\sn)$, for $t$ close to $0$, there exists $M_1>0$ such that $[h_t]\subset M_1B$.  Denote $F(s) = \int_{0}^{s}e^{-\frac{r^2}{2}} r^{n-1}dr$. By mean value theorem,
	\begin{equation*}
		|F(\rho_{[h_t]}(u)) - F(\rho_K(u))|\leq |F'(\theta)| |\rho_{[h_t]}(u)-\rho_K(u)|< M|F'(\theta)| |t|, 
	\end{equation*}
	where $M$ comes from Lemma \ref{lemma local 10} and $\theta$ is between $\rho_{[h_t]}(u)$ and $\rho_{K}(u)$. Since $[h_t]\subset M_1B$, we have $\theta\in (0,M_1]$. Therefore, by definition of $F$, we have $|F'(\theta)|$ is bounded from above by some constant that depends on $M_1$. Therefore, there exists $M_2>0$ such that 
	\begin{equation*}
		|F(\rho_{[h_t]}(u)) - F(\rho_K(u))|\leq M_2|t|.
	\end{equation*}
	Using dominated convergence theorem, together with Lemma \ref{lemma local 10}, we have
	\begin{equation*}
	\begin{aligned}
		\lim_{t\rightarrow 0} \frac{\gamma_n([h_K+tf])-\gamma_n(K)}{t} &= \frac{1}{(\sqrt{2\pi})^n} \int_\sn f(\alpha_K(u))e^{-\frac{\rho_K(u)^2}{2}}\frac{\rho_K(u)^n}{h_K(\alpha_K(u))}du\\
		& = \frac{1}{(\sqrt{2\pi})^n} \int_{\partial K} f(\nu_K(x))e^{-\frac{|x|^2}{2}}d\mathcal{H}^{n-1}(x)\\
		& = \frac{1}{(\sqrt{2\pi})^n} \int_{\sn} fdS_{\gamma_n, K}.
	\end{aligned}
	\end{equation*}
\end{proof}

The Gaussian surface area measure is weakly convergent with respect to Hausdorff metric.
\begin{theorem}
	Let $K_i\in \mathcal{K}_o^n$ such that $K_i$ converges to $K_0\in \mathcal{K}_o^n$ in Hausdorff metric. Then $S_{\gamma_n, K_i}$ converges to $S_{\gamma_n, K_0}$ weakly.
\end{theorem}
\begin{proof}
	Note that since $K_0\in \mathcal K_o^n$, there exists $C>0$ such that $\frac{1}{C}B \subset K_i\subset CB$ for sufficiently large $i$. Therefore, we have
	\begin{equation}
	\label{eq local 9090}
		e^{-\frac{\rho_{K_i}(u)^2}{2}} \rho_{K_i}^{n-1}(u) \rightarrow e^{-\frac{\rho_{K_0}(u)^2}{2}} \rho_{K_0}^{n-1}(u) \qquad\text{uniformly on }\sn.
	\end{equation}
	Let $g\in C(\sn)$. Since $\alpha_{K_i}\rightarrow \alpha_{K_0}$ almost everywhere on $\sn$ with respect to spherical Lebesgue measure, we also have
	\begin{equation}
	\label{eq local 9091}
		\frac{g(\alpha_{K_i}(u))}{u\cdot \alpha_{K_{i}}(u)} \rightarrow \frac{g(\alpha_{K_0}(u))}{u\cdot \alpha_{K_{0}}(u)}\qquad \text{almost everywhere on }\sn.
	\end{equation}
	
	By definition of $S_{\gamma_n,K}$, we have
	\begin{equation*}
	\begin{aligned}
			\int_{\sn} gdS_{\gamma_n, K_i} &= \frac{1}{(\sqrt{2\pi})^n} \int_{\partial K_i} g(\nu_{K_i}(x))e^{-\frac{|x|^2}{2}}d \mathcal{H}^{n-1}(x)\\ 
			& =\frac{1}{(\sqrt{2\pi})^n} \int_{\sn} g(\alpha_{K_i}(u)) e^{-\frac{\rho_{K_i}(u)^2}{2}} \frac{\rho_{K_i}(u)^n}{h_{K_i}(\alpha_{K_i}(u))}du\\
			& = \frac{1}{(\sqrt{2\pi})^n} \int_{\sn} \frac{g(\alpha_{K_i}(u))}{u\cdot \alpha_{K_i}(u)} e^{-\frac{\rho_{K_i}(u)^2}{2}} {\rho_{K_i}(u)^{n-1}}du\\
			& \rightarrow \frac{1}{(\sqrt{2\pi})^n} \int_{\sn} \frac{g(\alpha_{K_0}(u))}{u\cdot \alpha_{K_0}(u)} e^{-\frac{\rho_{K_0}(u)^2}{2}} {\rho_{K_0}(u)^{n-1}}du\\
			& = \int_{\sn} gdS_{\gamma_n, K_0},
	\end{aligned}
	\end{equation*}
	where the limit is due to \eqref{eq local 9090} and \eqref{eq local 9091}.
\end{proof}

By a simple calculation, it follows from the definition of Gaussian surface area measure that if $K\in \mathcal{K}_o^n$ is convex, then $S_{\gamma_n, K}$ is absolutely continuous with respect to surface area measure and 
\begin{equation*}
	dS_{\gamma_n, K } = \frac{1}{(\sqrt{2\pi})^n} e^{-\frac{|\nabla h_K|^2+h_K^2}{2}} dS_K.
\end{equation*}
If, in addition, the body $K$ is $C^2$, then $S_{\gamma_n, K}$ is absolutely continuous with respect to spherical Lebesgue measure and 
\begin{equation}
\label{eq local 9093}
	dS_{\gamma_n, K }(v)= \frac{1}{(\sqrt{2\pi})^n} e^{-\frac{|\nabla h_K|^2+h_K^2}{2}} \det (\nabla^2 h_K+h_KI)dv.
\end{equation}

When $P\in \mathcal{K}_o^n$ is a polytope with unit normal vectors $v_i$ with the corresponding faces $F_i$, the Gaussian surface area measure $S_{\gamma_n, P}$ is a discrete measure given by
\begin{equation*}
	S_{\gamma_n, P}(\cdot) = \sum_{i=1}^N c_i \delta_{v_i}(\cdot),  
\end{equation*}
where $c_i$ is given by 
\begin{equation*}
	c_i=\frac{1}{(\sqrt{2\pi})^n}\int_{F_i} e^{-\frac{|x|^2}{2}}d\mathcal{H}^{n-1}(x).
\end{equation*}

The classical Minkowski problem asks for the existence, uniqueness and regularity of a convex body $K$ whose surface area measure is prescribed. It has played a fundamental role, not only in convex geometric analysis, but also in PDE, differential geometry, functional analysis. Given this, it is natural to study the corresponding problem for Gaussian surface area measure, which we refer to as \emph{the Gaussian Minkowski problem}.

\textbf{The Gaussian Minkowski problem.} Given a finite Borel measure $\mu$, what are the necessary and sufficient conditions on $\mu$ so that there exists a convex body $K$ with $o\in \text{int}\,K$ such that
\begin{equation*}
	\mu  = S_{\lambda_n, K}?
\end{equation*}
If $K$ exists, to what extent is it unique?

It follows from \eqref{eq local 9093} that if $\mu = fdv$, then the Gaussian Minkowski problem is equivalent to the study of the following Monge-Amp\`{e}re type equation on $\sn$:
\begin{equation*}
	\frac{1}{(\sqrt{2\pi})^n} e^{-\frac{|\nabla h|^2+h^2}{2}}\det (\nabla^2 h+hI) = f.
\end{equation*}

It is well-known that the classical surface area measure $S_K$ when viewed as a map from the set of convex bodies to the set of Borel measures on $\sn$ is a valuation. In fact, Haberl-Parapatits \cite{MR3176613} gave a valuation characterization of surface area measure. Valuation theory plays an important role in convex geometry, see, \emph{e.g.}, \cite{MR2966660, MR2772547, MR2680490,MR2668553}. Similar to the proof for that of surface area measure, it is not hard to see that Gaussian surface area measure is also a valuation; that is, if $K$ and $L$ are two convex bodies such that $K\cup L$ is also a convex body, then
\begin{equation*}
	S_{\gamma_n, K\cup L} + S_{\gamma_n, K\cap L} = S_{\gamma_n, K}+S_{\gamma_n, L}.
\end{equation*}
It is of great interest to see if there is a valuation characterization of Gaussian surface area measure.

\subsection{The normalized problem and its solution}
\label{subsection normalized solution}

Motivated by the work of Haberl-Lutwak-Yang-Zhang \cite{MR2652213}, we derive the solution to the following normalized version of the even Gaussian Minkowski problem.

\begin{theorem}
\label{thm solution to the normalized problem}
	Suppose $\mu$ is an finite even Borel measure not concentrated in any closed hemisphere. Then, for each $0<\alpha<\frac{1}{n}$, there exists an $o$-symmetric convex body $K$ such that 
	\begin{equation}
	\label{eq local 000011}
		\mu = \frac{S_{\gamma_n, K}}{\gamma_n(K)^{1-\alpha}}.
	\end{equation}
\end{theorem}

Our approach to the normalized problem is variational and involves the following optimization problem:
\begin{equation}
\label{eq local 00000}
	\sup \{ \Gamma(f): f\in C_e^+(\sn)\},
\end{equation}
where $\Gamma: C^+(\sn)\rightarrow \mathbb{R}$ is given by
\begin{equation*}
	\Gamma(f) = \frac{1}{\alpha}\gamma_n([f])^\alpha- \int_\sn fd\mu.
\end{equation*}
\begin{lemma}
\label{lemma local 00000}
	Let $0<\alpha<\frac{1}{n}$. If an even function $f_0$ is a maximizer to the optimization problem \eqref{eq local 00000}, then $f_0$ must be the support function of an $o$-symmetric convex body; that is, there exists an $o$-symmetric convex body $K_0$ such that $f_0=h_{K_0}$. Moreover, $K_0$ satisfies \eqref{eq local 000011}. 
\end{lemma}
\begin{proof}
	Note that for each $f\in C_e^+(\sn)$, by the definition of Wulff shape, we have
	\begin{equation*}
		\Gamma(f)\geq \Gamma(h_{[f]}).
	\end{equation*}
	Therefore, the maximizer $f_0$ must be the support function of some $o$-symmetric convex body $K_0$. 
	
	We now use the variational formula \eqref{eq local 7002} to establish that $K_0$ satisfies \eqref{eq local 000011}.
	
	Towards this end, for each $g\in C_{e}^+(\sn)$, consider the one-parameter family 
	\begin{equation*}
		K_t= [h_{K_0}+tg].
	\end{equation*}
	Since $f_0=h_{K_0}$ is a maximizer, we have
	\begin{equation*}
		0 = \left.\frac{d}{dt}\right|_{t=0} \Gamma(h_{K_t}) = \gamma_n(K_0)^{\alpha-1} \int_{\sn} gdS_{\gamma_n, K_0}-\int_\sn gd\mu,
	\end{equation*}
	where in the second equality, we used the variational formula \eqref{eq local 7002}. Note that the above equation holds for every $g\in C_e^+(\sn)$. Therefore, we conclude that $K_0$ satisfies \eqref{eq local 000011}.
\end{proof}

For simplicity, when no confusion arises, we will write $\Gamma(K)$ in place of $\Gamma(h_K)$.

\begin{lemma}
\label{lemma local 00001}
	Let $0<\alpha<\frac{1}{n}$. For sufficiently small $r>0$, we have $\Gamma(rB)>0$. 
\end{lemma}
\begin{proof}
	By definition of $\Gamma$, we have
	\begin{equation*}
	\begin{aligned}
		\Gamma(rB) &= \frac{1}{\alpha} \left(\frac{1}{(2\pi)^\frac{n}{2}}\int_{rB}e^{-\frac{|x|^2}{2}}dx\right)^\alpha - \int_{\sn }rd\mu\\
		&= \frac{1}{\alpha}\left(\frac{n\omega_n}{(2\pi)^\frac{n}{2}}\int_{0}^r e^{-\frac{t^2}{2}}t^{n-1}dt\right)^\alpha - r|\mu|\\
		& = r \left[\frac{1}{\alpha }\left(\frac{n\omega_n}{(2\pi)^\frac{n}{2}}\right)^\alpha \left(\frac{\int_{0}^r e^{-\frac{t^2}{2}}t^{n-1}dt}{r^{\frac{1}{\alpha}}}\right)^\alpha -|\mu|\right].
	\end{aligned}
	\end{equation*}
	It follows from simple computation that when $0<\alpha<\frac{1}{n}$, we have 
	\begin{equation*}
		\lim_{r\rightarrow 0^+}\frac{\int_{0}^r e^{-\frac{t^2}{2}}t^{n-1}dt}{r^{\frac{1}{\alpha}}} = \lim_{r\rightarrow 0^+}\frac{\alpha e^{-\frac{r^2}{2}}r^{n-1}}{r^{\frac{1}{\alpha}-1}} = \alpha \lim_{r\rightarrow 0^+} r^{n-\frac{1}{\alpha}} = \infty.
	\end{equation*}
	The desired result is therefore established.
\end{proof}

We are now ready to give a proof to Theorem \ref{thm solution to the normalized problem}. 
\begin{proof}[Proof of Theorem \ref{thm solution to the normalized problem}]

	By Lemma \ref{lemma local 00000}, it suffices to show that a maximizer to the optimization problem \eqref{eq local 00000} exists. We assume that $K_i$ is a sequence of $o$-symmetric convex bodies and 
	\begin{equation}
	\label{eq local 00003}
		\lim_{i\rightarrow \infty} \Gamma(K_i) = \sup \{ \Gamma(f): f\in C_e^+(\sn)\}>0,
	\end{equation}
	where the last inequality follows from Lemma \ref{lemma local 00001}. 
	
	Choose $r_i>0$ and $u_i\in \sn$ such that $r_iu_i\in K_i$ and 
	\begin{equation*}
		r_i = \max_{u\in \sn } \rho_{K_i}(u).
	\end{equation*}
	It is simple to notice that $K_i\subset r_i B$. We claim that $r_i$ is a bounded sequence. Otherwise, by taking a subsequence, we may assume that $\lim_{i\rightarrow \infty} r_i = \infty$. Since $K_i\subset r_i B$, we have
	\begin{equation*}
		\Gamma(K_i)\leq \frac{1}{\alpha} \gamma_n(r_iB)^\alpha - \int_\sn h_{K_i}d\mu.
 	\end{equation*}
 	Since $r_iu_i\in K_i$, we have by the definition of support function that 
 	\begin{equation*}
 		h_{K_i}(v)\geq r_i|v\cdot u_i|.
 	\end{equation*}
 	Therefore, we have
 	\begin{equation*}
 		\Gamma(K_i)\leq \frac{1}{\alpha} \gamma_n(r_iB)^\alpha - r_i\int_\sn |v\cdot u_i|d\mu.
 	\end{equation*}
 	By the fact that $\mu$ is not concentrated in any closed hemisphere, we may find $c_0>0$ such that 
 	\begin{equation*}
 		\int_\sn |v\cdot u_i|d\mu\geq c_0. 
 	\end{equation*}
 	Therefore,
 	\begin{equation*}
 	\begin{aligned}
 		\Gamma(K_i)&\leq \frac{1}{\alpha} \gamma_n(r_iB)^\alpha -r_i c_0\\
 		& = \frac{1}{\alpha}\left(\frac{n\omega_n}{(2\pi)^\frac{n}{2}}\int_{0}^{r_i} e^{-\frac{t^2}{2}}t^{n-1}dt\right)^\alpha - r_ic_0\\
 		& \rightarrow -\infty,
 	\end{aligned}
 	\end{equation*}
 	as $i\rightarrow\infty$, since the integral $\int_{0}^\infty e^{-\frac{t^2}{2}}t^{n-1}dt$ is convergent. But this is a contradiction to $K_i$ being a maximizing sequence and \eqref{eq local 00003}. Therefore, the sequence of convex bodies $K_i$ is uniformly bounded. We may therefore use Blaschke selection theorem and assume (by taking a subsequence) that $K_i$ converges in Hausdorff metric to a compact convex $o$-symmetric set $K_0$. Note that by the continuity of the Gaussian volume with respect to the Hausdorff metric, definition of $\Gamma$, and \eqref{eq local 00003}, we have
 	\begin{equation*}
 		\frac{1}{\alpha}\gamma_n(K_0)^\alpha \geq \Gamma(K_0) = \lim_{i\rightarrow \infty} \Gamma(K_i)>0.
 	\end{equation*}
 	This, when combined with the fact that $K_0$ is $o$-symmetric, implies that $K_0$ contains the origin as its interior point. Therefore, the convex body $K_0$ (or, its support function $h_{K_0}$) is a maximizer to the optimization \eqref{eq local 00000}.
\end{proof}

It is of great interest to ask whether the convex body satisfying \eqref{eq local 000011} is uniquely determined. 

\section{Isoperimetric inequalities in Gaussian probability space}
In this section, we recall the Ehrhard inequality and several of its consequences. 

The Ehrhard inequality was shown by Ehrhard \cite{MR745081} when both Borel sets involved are convex,  by Lata\l a \cite{MR1389763} when only one of the sets is assumed to be convex, and more recently by Borell \cite{MR2030108} when neither set is required to be convex. However, for the purpose of this paper, only Ehrhard's original version is required.
\begin{theorem}[Ehrhard inequality] Let $K, L$ be two convex bodies in $\rn$. 
	For $0<t<1$, we have
	\begin{equation*}
		\Phi^{-1}(\gamma_n((1-t)K+tL))\geq (1-t)\Phi^{-1}(\gamma_n(K))+t\Phi^{-1}(\gamma_n(L)).
	\end{equation*}
	Here,
	\begin{equation}
	\label{eq local 203}
		\Phi(x) = \frac{1}{\sqrt{2\pi}}\int_{-\infty}^{x} e^{-\frac{t^2}{2}}dt.
	\end{equation}
	Moreover, equality holds if and only if $K=L$.
\end{theorem}

The equality condition in the above lemma was shown by Ehrhard \cite{MR850753}. More recently, for more general versions of Ehrhard inequality, the equality condition has been settled by Shenfeld and van Handel \cite{MR3804680}.

The following lemma is a direct consequence of Ehrhard inequality. See also Borell \cite{MR404559} for a characterization of log-concave measures.

\begin{lemma}
\label{lemma local 101}
	Let $K$ and $L$ be two convex bodies in $\rn$. For $0<t<1$, we have
	\begin{equation}
	\label{eq log concavity}
		\gamma_n((1-t)K+tL)\geq \gamma_n(K)^{1-t}\gamma_n(L)^t,
	\end{equation}
	with equality if and only if $K=L$.
\end{lemma}

Using the variational formula \eqref{eq local 7002}, we obtain the following Minkowski-type inequality.
\begin{lemma}
Let $K$ and $L$ be two convex bodies in $\rn$. We have
\begin{equation*}
		\int_\sn h_L-h_KdS_{\gamma_n, K}\geq \gamma_n(K)\log \frac{\gamma_n(L)}{\gamma_n(K)},
\end{equation*}
	with equality if and only if $K=L$.
\end{lemma}
\begin{proof}
	Lemma \ref{lemma local 101} implies that the function $g:[0,1]\rightarrow \mathbb{R}$ given by
	\begin{equation*}
		g(t) = \log \gamma_n((1-t)K+tL)
	\end{equation*}
	is concave. Therefore, the slope of the tangent line at $t=0$ is no smaller than the slope of the secant line joining $(0,g(0))$ and $(1,g(1))$. Using \eqref{eq local 7002} to compute $g'(0)$, we immediately arrive at the desired inequality. 
	
	If equality holds, then $g(t)$ is a linear function. Thus, equality holds in \eqref{eq log concavity}, which then implies that $K=L$. 
	
\end{proof}

An immediate consequence is
\begin{lemma}
\label{coro first variation of log concavity}
	Let $K$ and $L$ be two convex bodies in $\rn$. If $\gamma_n(K)=\gamma_n(L)$, then
	\begin{equation}
	\label{eq local 7007}
		\int_{\sn} h_L dS_{\gamma_n,K} \geq  \int_{\sn} h_K dS_{\gamma_n,K},
	\end{equation}
	with equality if and only $K=L$.
\end{lemma}

The Ehrhard inequality also implies the following isoperimetric inequality in Gaussian probability space. See, for example, \cite{MR1957087}.
\begin{theorem}[Gaussian isoperimetric inequality]
	Let $K$ be a convex body in $\mathbb{R}^n$. Then,
	\begin{equation*}
|S_{\gamma_n,K}|\geq \varphi(\Phi^{-1}(\gamma_n(K))),
	\end{equation*}
	where $\varphi(t) = (\sqrt{2\pi})^{-1}\exp(-t^2/2)$ and $\Phi$ is as given in \eqref{eq local 203}. 
\end{theorem}

The Guassian isoperimetric inequality is a consequence of Ehrhard inequality. A direct consequence of the Gaussian isoperimetric inequality is the following.
\begin{coro}
\label{coro local 1001}
	If $K$ is a convex body in $\rn$ such that $\gamma_n(K)=1/2$, then $|S_{\gamma_n,K}|\geq\frac{1}{\sqrt{2\pi}}$.
\end{coro}

\section{Uniqueness of solution}

In this section, we will show that the solution to the Gaussian Minkowski problem is unique if one restricts the solution set to bodies with sufficiently big Gaussian volume.

\begin{lemma}
\label{lemma equal volume}
	Suppose $K, L\in \mathcal{K}_o^n$ and $K, L$ both solve the Gaussian Minkowski problem; i.e.,
	\begin{equation}
	\label{eq local 209}
		S_{\gamma_n, K} = S_{\gamma_n, L}= \mu.
	\end{equation}
	If $\gamma_n(K),\gamma_n(L)\geq 1/2$, then,
	\begin{equation*}
		\gamma_n(K)=\gamma_n(L).
	\end{equation*}
\end{lemma}
\begin{proof}
	For simplicity, we write $\Psi = \Phi^{-1}$ where $\Phi$ is given in \eqref{eq local 203}. Then, Ehrhard inequality says
	\begin{equation}
	\label{eq local 206}
		\Psi (\gamma_n((1-t)K+tL))\geq (1-t) \Psi (\gamma_n(K))+t\Psi(\gamma_n(L)),
	\end{equation}
	with equality if and only if $K=L$. Notice that $\Psi$ is $C^\infty$ and strictly monotonically increasing. Utilizing Theorem \ref{thm variational formula}, we take the first derivative of \eqref{eq local 206} at $t=0$ and get
	\begin{equation}
	\label{eq local 207}
		\Psi'(\gamma_n(K))\int_{\sn} h_L-h_K dS_{\gamma_n, K} \geq \Psi(\gamma_n(L))-\Psi(\gamma_n(K)).
 	\end{equation}
 	Switching the role of $K$ and $L$, we get
 	\begin{equation}
 	\label{eq local 208}
 		\Psi'(\gamma_n(L))\int_{\sn} h_K-h_L dS_{\gamma_n, L} \geq \Psi(\gamma_n(K))-\Psi(\gamma_n(L)).
 	\end{equation}
 	By \eqref{eq local 209}, we have
 	\begin{equation}
 	\label{eq local 210}
 		\Psi'(\gamma_n(L))\int_{\sn} h_L-h_K dS_{\gamma_n, K} \leq \Psi(\gamma_n(L))-\Psi(\gamma_n(K)).
 	\end{equation}
 	from \eqref{eq local 208}. Since $\Psi'>0$, \eqref{eq local 207} and \eqref{eq local 210} imply that
 	\begin{equation*}
 		\frac{\Psi(\gamma_n(L))-\Psi(\gamma_n(K))}{\Psi'(\gamma_n(L))}\geq \frac{\Psi(\gamma_n(L))-\Psi(\gamma_n(K))}{\Psi'(\gamma_n(K))},
 	\end{equation*}
 	or, equivalently,
 	\begin{equation}
 	\label{eq local 211}
 		\left(\Psi'(\gamma_n(K))-\Psi'(\gamma_n(L))\right)\left(\Psi(\gamma_n(K))-\Psi(\gamma_n(L))\right)\leq 0.
 	\end{equation}
 	By the definition of $\Psi$ and chain rule, we can compute
 	\begin{equation*}
 		\Psi'(x)= \sqrt{2\pi}e^{\frac{\Psi(x)^2}{2}}.
 	\end{equation*}
 	Since $\Psi$ is strictly increasing on $[1/2,1]$, this implies that $\Psi'$ is also strictly increasing. This, when combined with the fact that $\Psi$ is strictly increasing, shows that
 	\begin{equation*}
 		(\Psi'(a)-\Psi'(b))(\Psi(a)-\Psi(b))\geq 0,
 	\end{equation*}
 	with equality if and only if $a=b$. Equation \eqref{eq local 211} now gives us the desired result.
\end{proof}

We are now ready to prove the uniqueness part of the Gaussian Minkowski problem when we restrict to the set of convex bodies whose Gaussian measure is no smaller than $1/2$.
\begin{theorem}
\label{thm uniqueness}
	Suppose $K, L\in \mathcal{K}_o^n$ and $K, L$ both solve the Gaussian Minkowski problem; i.e.,
	\begin{equation*}
		S_{\gamma_n, K} = S_{\gamma_n, L}= \mu.
	\end{equation*}
	If $\gamma_n(K),\gamma_n(L)\geq 1/2$, then $K=L$.
\end{theorem}
\begin{proof}
	By Lemma \ref{lemma equal volume}, we have $\gamma_n(K)=\gamma_n(L)$. By Corollary \ref{coro first variation of log concavity} and the fact that $S_{\gamma_n, K} = S_{\gamma_n, L}$, we conclude that equality holds in \eqref{eq local 7007} and therefore by the equality condition, we have $K=L$.
\end{proof}

\section{Existence of $o$-symmetric solutions}
\label{section existence}
For the rest of the paper, we are going to prove existence results regarding the Gaussian Minkowski problem. For this purpose, we shall restrict ourselves to the $o$-symmetric case; that is, when the given data ($\mu$ or in the smooth case, its density $f$) is even and the potential solution set is restricted to $\mathcal{K}_e^n$.

In this section, we will first prove the existence result when the given data is sufficiently smooth and everywhere positive. To do that, some \emph{a-priori} estimates are required. At the end of the section, an approximation argument will be deployed to get the solution when the given data is a measure.

\subsection{$C^0$ estimate}

\begin{lemma}
\label{lemma C0 lower bound}
	There exists a constant $c>0$ such that if $K\in \mathcal{K}_e^n$ and $\gamma_n(K)\geq \frac{1}{2}$, then its support function $h_K$ is bounded from below by $c$ on $\sn$.
\end{lemma}
\begin{proof}
	We argue by contradiction and assume that there exists $K_i\in \mathcal{K}_e^n$ with $\gamma_n(K_i)\geq \frac{1}{2}$ and $v_i\in \sn$ such that $h_i:=h_{K_i}(v_i)\rightarrow 0$. Then, by definition of support function, we have
	\begin{equation*}
		K_i\subset \{x\in \rn: |x\cdot v_i|\leq h_i\}.
	\end{equation*}
	Therefore,
	\begin{equation*}
		\gamma_n(K_i)\leq \gamma_n(\{x\in \rn: |x\cdot v_i|\leq h_i\})\rightarrow 0,
	\end{equation*}
	which contradicts with the given condition that $\gamma_n(K_i)\geq \frac{1}{2}$.
\end{proof}

\begin{lemma}
\label{lemma far away 0}
	Let $K_i$ be a sequence of convex bodies in $\mathcal{K}_e^n$ and $r_i=|h_{K_i}|_{C^\infty}$. If $\lim_{i\rightarrow \infty} r_i=\infty$, then for each $0<c\leq 1$ and $v\in \sn$, we have
	\begin{equation*}
		\lim_{i\rightarrow \infty} \int_{\omega_i} e^{-\frac{|x|^2}{2}}d\mathcal{H}^{n-1}(x)=0,
	\end{equation*}
	where
	\begin{equation*}
		\omega_i =\{x\in \partial K_i: x\cdot v>cr_i\}.
	\end{equation*}
\end{lemma}
\begin{proof}
	We first recall that based on Cauchy's surface area formula, if $K\subset L$, then $\mathcal{H}^{n-1}(\partial K)\leq \mathcal{H}^{n-1}(\partial L)$.
	
	For simplicity, we write $B_i$ for the centered ball of radius $r_i$. For each $j=1,2,\dots$, write
	\begin{equation*}
		L_{i,j}= B_i\cap \{x\in \rn: j\leq x\cdot v\leq j+1\}.
	\end{equation*}
	Note that $L_{i,j}$ can be contained in a cylinder with the base of an $(n-1)$-dimensional ball of radius $r_i$ and unit height. Thus,
	\begin{equation}
	\label{eq local 7010}
		\mathcal{H}^{n-1}(\partial L_{i,j})\leq \mathcal{H}^{n-1}(\partial (B_i\cap \mathbb{R}^{n-1})\times [0,1])\leq c(n) r_i^{n-1}.
	\end{equation}
	Here and in the rest of the proof, we frequently use symbols such as $c(n)$ to denote nonessential constants that only depend on the dimension.
	
	Let $\Gamma_{i,j} = K_i\cap \{x\in \rn: j\leq x\cdot v\leq j+1\}$. Note that
	\begin{equation}
	\label{eq local 7011}
		\omega_i \subset \bigcup_{j=\lfloor cr_i\rfloor}^\infty \partial \Gamma_{i,j}.
 	\end{equation}
 	For each individual $\partial \Gamma_{i,j}$, we have
 	\begin{equation*}
 		\int_{\partial \Gamma_{i,j}} e^{-\frac{|x|^2}{2}}d\mathcal{H}^{n-1}(x)\leq e^{-\frac{j^2}{2}}\mathcal{H}^{n-1}(\partial \Gamma_{i,j})\leq e^{-\frac{j^2}{2}}\mathcal{H}^{n-1}(\partial L_{i,j}),
 	\end{equation*}
 	where in the last inequality, we used the fact that $\Gamma_{i,j}\subset L_{i,j}$. Combining with \eqref{eq local 7010} and \eqref{eq local 7011}, we have
 	\begin{equation*}
 	\begin{aligned}
 		\int_{\omega_i}e^{-\frac{|x|^2}{2}}d\mathcal{H}^{n-1}(x)&\leq \sum_{j=\lfloor cr_i\rfloor}^\infty \int_{\partial \Gamma_{i,j}}e^{-\frac{|x|^2}{2}}d\mathcal{H}^{n-1}(x)\\
 		&\leq c(n)r_i^{n-1}\sum_{j=\lfloor cr_i\rfloor}^\infty e^{-\frac{j^2}{2}}\\
 		&\leq c(n)r_i^{n-1} \sum_{j=\lfloor cr_i\rfloor}^\infty e^{-j}\\
 		&\leq c(n)r_i^{n-1}e^{-\lfloor cr_i\rfloor}\sum_{j=0}^\infty e^{-j}\\
 		&\rightarrow 0,
 	\end{aligned}
 	\end{equation*}
 	as $i\rightarrow \infty$, since $r_i\rightarrow \infty$.

\end{proof}

\begin{lemma}
\label{lemma key}
	Suppose $K_i$ is a sequence of convex bodies in $\mathcal{K}_e^n$. For simplicity, write $\mu_i = S_{\gamma_n, K_i}$. 
	If $|h_{K_i}|_{C^\infty}\rightarrow \infty$ and $|\mu_i|>\frac{1}{C}$ for some $C>0$, then for every $\delta>0$, there exists $v\in \sn$ and $N>0$ such that
	\begin{equation*}
		|\mu_i|<\mu_i(\overline{\xi_{v,\delta}})+\delta,
	\end{equation*}
	for each $i>N$. Here $\xi_{v,\delta}$ is given by
	\begin{equation*}
		\xi_{v,\delta} = \{u\in \sn:|u\cdot v|<\delta\}. 
	\end{equation*}
\end{lemma}
\begin{proof}
	By John's theorem, there exists $o$-symmetric ellipsoid $E_i$:
\begin{equation*}
	E_i = \left\{ x\in \mathbb{R}^n:\frac{|x\cdot e_{i,1}|^2}{r_{i,1}^2}+\cdots + \frac{|x\cdot e_{i,n}|^2}{r_{i,n}^2}\leq 1\right\}
\end{equation*}
with $r_{i,1}\geq r_{i,2}\geq \dots\geq r_{i,n}$ such that $E_i\subset K_i\subset \sqrt{n} E_i$. Since $|h_i|_{C^{\infty}}\rightarrow \infty$, we have $r_{i,1}\rightarrow \infty$.

By taking a subsequence, we may assume
	\begin{equation*}
		0\leq \lim_{i\rightarrow\infty} \frac{r_{i,j}}{r_{i,j-1}} = a_j\leq 1, \text{ for } j=2, \dots n.
	\end{equation*}
	Define $a_1=1$. We argue that there exists $j\in \{1,\dots, n\}$ such that $a_j=0$. If not, then there exists $0<c\leq 1$ such that $cr_{i,1}< r_{i,n}\leq r_{i,1}$. Choose an orthonormal basis $v_{1},\dots, v_n$. By definition of $r_{i,n}$, we see that
	\begin{equation}
	\label{eq local 70121}
		(\sqrt{2\pi})^n|S_{\gamma_n, K_i}| = \int_{\partial K_i} e^{-\frac{|x|^2}{2}}d\mathcal{H}^{n-1}(x)\leq \sum_{j=1}^n \int_{\omega_{i,j}} e^{-\frac{|x|^2}{2}}d\mathcal{H}^{n-1}(x),
	\end{equation}
	where
	\begin{equation*}
		\omega_{i,j} = \{x\in \partial K_i: |x\cdot v_j|>cr_{i,1}\}.
	\end{equation*}
	Lemma \ref{lemma far away 0} and \eqref{eq local 70121} now implies that $\lim_{i\rightarrow\infty} |S_{\gamma_n,K_i} |=0$, which is a contradiction to the uniform lower bound of $|\mu_i|$.
	
	Now, let $s=\min\{i-1: a_i=0\}$.  Then there exists $0<c\leq 1$ such that
	\begin{equation}
	\label{eq local 80151}
	r_{i,1}\geq r_{i,2}\geq\dots\geq r_{i,s}\geq cr_{i,1},	
	\end{equation}
and
\begin{equation}
\label{eq local 80201}
	\lim_{i\rightarrow \infty} \frac{r_{i,j}}{r_{i,s}}=0, \text{ for each }j>s.
\end{equation}
 By possibly taking another subsequence, we may assume
 \begin{equation}
 \label{eq local 80401}
 	\lim_{i\rightarrow \infty} e_{i,j}= e_j, \text{ for } j =1, 2, \dots, n.
 \end{equation}
 Here $e_j$ is an orthonormal basis in $\rn$.

 Choose $\tau_n,\varepsilon_1 \in (0,1)$ such that
 \begin{equation*}
 	\frac{n}{c^2}(2\tau_n+\sqrt{n}\varepsilon_1)^2\leq 1,
 \end{equation*}
 where $c>0$ is from \eqref{eq local 80151}.

 Let
	\begin{equation*}
		\begin{aligned}
			\Omega_i &= \{x\in \partial K_i: |x\cdot e_{j}|\leq \tau_n r_{i,s}, j=1, 2,\dots, s\}\\
			\eta_{i,j} & = \{x\in \partial K_i: |x\cdot e_{j}|>\tau_n {r_{i,s}}\}.
		\end{aligned}
	\end{equation*}
	Then, it is simple to see that
	\begin{equation*}
		\partial K_i = \Omega_i \cup \left(\cup_{j=1}^s \eta_{i,j}\right).
	\end{equation*}
	Recall that
	\begin{equation}
	\label{eq local 80041}
		(\sqrt{2\pi})^n|S_{K_i, \gamma_n}| = \int_{\partial K_i} e^{-\frac{|x|^2}{2}}d\mathcal{H}^{n-1}(x) = \int_{\Omega_i} e^{-\frac{|x|^2}{2}}d\mathcal{H}^{n-1}(x) + \sum_{j=1}^s \int_{\eta_{i,j}} e^{-\frac{|x|^2}{2}}d\mathcal{H}^{n-1}(x).
	\end{equation}
		By Lemma \ref{lemma far away 0} and \eqref{eq local 80151},
	\begin{equation}
	\label{eq local 80901}
		\lim_{i\rightarrow \infty} \sum_{j=1}^s \int_{\eta_{i,j}} e^{-\frac{|x|^2}{2}}d\mathcal{H}^{n-1}(x)=0.
	\end{equation}
	
	It remains to estimate
	\begin{equation*}
		 \int_{\Omega_i} e^{-\frac{|x|^2}{2}}d\mathcal{H}^{n-1}(x).
	\end{equation*}
	Take $x\in \Omega_i$. Set $z = (x\cdot e_{i,1} \dots, x\cdot e_{i, s-1},  x\cdot e_{i,s}+\tau_n r_{i,s},0,\dots, 0)$, where the coordinates are under $e_{i,1},\dots e_{i,n}$. By \eqref{eq local 80401}, there exists $N_1>0$ such that for each $i>N_1$, we have
	\begin{equation}
	\label{eq local 80701}
		|e_{i,j}-e_j|<\varepsilon_1, \text{ for each } j = 1,\dots, n.
	\end{equation}
	We claim that for $i>N_1$, we have $z\in E_i$. Indeed, by the triangle inequality, the definition of $\Omega_i$, \eqref{eq local 80701}, the fact that $K_i\subset \sqrt{n}E_i$, \eqref{eq local 80151}, and the choice of $\tau_n$ and $\varepsilon_1$,
	\begin{equation*}
		\begin{aligned}
			\sum_{j=1}^n\frac{|z\cdot e_{i,j}|^2}{r_{i,j}^2} &= \sum_{j=1}^{s-1} \frac{|x\cdot e_{i,j}|^2}{r_{i,j}^2} + \frac{|x\cdot e_{i,s}+\tau_n r_{i,s}|^2}{r_{i,s}^2}\\
			&\leq \sum_{j=1}^{s-1}\frac{(|x\cdot e_{j}|+|x\cdot (e_{i,j}-e_j)|)^2}{r_{i,j}^2}+\frac{(|x\cdot e_{s}+\tau_n r_{i,s}|+|x\cdot (e_{i,s}-e_s)|)^2}{r_{i,s}^2}
			\\
			&\leq \sum_{j=1}^{s-1}\frac{(\tau_n r_{i,s}+\sqrt{n}r_{i,1}\varepsilon_1)^2}{r_{i,j}^2}+ \frac{(2\tau_n r_{i,s}+\sqrt{n}r_{i,1}\varepsilon_1)^2}{r_{i,s}^2}\\
			&\leq \sum_{j=1}^{s-1} (\tau_n+\sqrt{n}\varepsilon_1)^2\frac{r_{i,1}^2}{r_{i,j}^2}+ (2\tau_n + \sqrt{n}\varepsilon_1)^2 \frac{r_{i,1}^2}{r_{i,s}^2}\\
			&\leq \frac{1}{c^2} n (2\tau_n+\varepsilon_1)^2\leq 1.
		\end{aligned}
	\end{equation*}
	Therefore, $z\in E_i\subset K_i$.

	Note that by definition of $z$ and the fact that $K_i\subset \sqrt{n}E_i$,
	\begin{equation*}
		d(x+ \tau_n r_{i,s}e_{i,s}, z)\leq \left(\sum_{j=s+1}^n |x\cdot e_{i,j}|^2\right)^\frac{1}{2}\leq \left(\sum_{j=s+1}^n nr_{i,j}^2\right)^\frac{1}{2}.
	\end{equation*}
	 According to \eqref{eq local 80201}, we have
	\begin{equation}
			\label{eq local 40221}
		\lim_{i\rightarrow \infty} \frac{d(x+ \tau_n r_{i,s}e_{i,s}, z)}{r_{i,s}}=0.
	\end{equation}
	
	Notice also that by definition of the Gauss map $\nu_{K_i}(x)$, we have
	\begin{equation*}
		\nu_{K_i}(x)\cdot (x-z)\geq 0,
	\end{equation*}
	which is equivalent to
	\begin{equation*}
		\nu_{K_i}(x)\cdot \left(\frac{x-z+\tau_n r_{i,s}e_{i,s}}{\tau_nr_{i,s}}-e_{i,s}\right)\geq 0.
	\end{equation*}
	By \eqref{eq local 40221}, for each $\delta>0$, there exists $N_2>N_1$ such that  for every $i>N_2$, we have
	\begin{equation*}
		\nu_{K_i}(x)\cdot e_{i,s}<\delta.
	\end{equation*}
	By symmetry, we yield
	\begin{equation*}
		|\nu_{K_i}(x)\cdot e_{i,s}|<\delta.
	\end{equation*}
	Since $e_{i,s}\rightarrow e_s$, there exists $N_3>N_2$ such that for every $i>N_3$, we have
	\begin{equation*}
		|\nu_{K_i}(x)\cdot e_{s}|<2\delta.
	\end{equation*}
	This, implies that
	\begin{equation*}
				\frac{1}{(\sqrt{2\pi})^n}\int_{\Omega_i} e^{-\frac{|x|^2}{2}}d\mathcal{H}^{n-1}(x) \leq \mu_i(\xi_{e_s,2\delta}),
	\end{equation*}
	This, \eqref{eq local 80041} and \eqref{eq local 80901} imply that for each $\delta>0$, there exists $N>0$ such that for each $i>N$, we have
	\begin{equation*}
		|\mu_i|\leq \mu_i(\overline{\xi_{e_s,\delta}})+\delta.
	\end{equation*}
\end{proof}

The following lemma contains the desired $C^0$ estimate.
\begin{lemma}
\label{lemma C0}
	Suppose the support function of $K\in \mathcal{K}_e^n$ is $C^2$ and  satisfies
	\begin{equation*}
		\frac{1}{(\sqrt{2\pi})^n}e^{-\frac{|\nabla h|^2+h^2}{2}}\det (\nabla^2 h +h\delta_{ij})=f,
	\end{equation*}
	and $\gamma_n(K)>1/2$. If there exists $C>0$ such that $\frac{1}{C}<f<C$, then there exists $C'>0$ such that $1/C'<h_K<C'$.
\end{lemma}

\begin{proof}
	The lower bound for $h_K$ comes from the assumption that $\gamma_n(K)>\frac{1}{2}$ and Lemma \ref{lemma C0 lower bound}.
	
	For the upper bound, assume for the sake of contradiction that there exists a sequence of $C^2$ $o$-symmetric convex bodies $K_i$ with $\gamma_n(K_i)>1/2$ and
	\begin{equation*}
		1/C<f_i=\frac{1}{(\sqrt{2\pi})^n}e^{-\frac{|\nabla h_i(v)|^2+h_i^2(v)}{2}}\det (\nabla^2 h_i (v) + h_{i}(v)I)<C,
	\end{equation*}
	but $|h_{i}|_{C^{\infty}}\rightarrow \infty$. Here we abbreviated $h_{K_i}$ by $h_{i}$.
	
	Write $\mu_i =f_idv$.

By Lemma \ref{lemma key}, for every $\delta>0$, there exists $v\in\sn$ and $N>0$ such that for each $i>N$, we have
\begin{equation}
\label{eq local 0010}
	|\mu_i|<\mu_i(\overline{\xi_{v, \delta}})+\delta.
\end{equation}
Note that there exists $c>0$ such that $\mathcal{H}^{n-1}(\overline{\xi_{v, \delta}})<c\delta$, which when combined with \eqref{eq local 0010} and the uniform upper bound of $f_i$, shows
\begin{equation*}
	|\mu_i|<cC\delta+\delta.
\end{equation*}
When $\delta$ is small enough, this is a contradiction to the uniform lower bound of $f_i$.
\end{proof}

\subsection{Higher order \emph{a-priori} estimate}
\begin{lemma}[a-priori estimate]
\label{lemma apriori estimate}
	Let $0<\alpha<1$. Suppose $f\in C^{2,\alpha}(\sn)$ and there exists $C>0$ such that $\frac{1}{C}<f<C$ and $|f|_{C^{2,\alpha}}<C$. If the support function of  $K\in \mathcal{K}_e^n$ is $C^{4,\alpha}$ and satisfies
	\begin{equation}\label{equaz1}
		\frac{1}{(\sqrt{2\pi})^n}e^{-\frac{|\nabla h|^2+h^2}{2}}\det (\nabla^2 h+ hI)=f
	\end{equation}
	and $\gamma_n(K)>\frac{1}{2}$, then there exists $C'>0$ that only depends on $C$ such that
	\begin{enumerate}
		\item $\frac{1}{C'}<\sqrt{|\nabla h_K|^2+ h_K^2}<C'$
		\item $\frac{1}{C'}I < (\nabla^2 h_K + h_K I )<C'I$
		\item $|h_K|_{C^{4,\alpha}}<C'$.
	\end{enumerate}	
\end{lemma}
\begin{proof}

\begin{enumerate}
	\item Note that by Lemma \ref{lemma C0}, there exists $C'>0$ such that $1/C'<h_K<C$, or, in another word, $1/C'B\subset K\subset C'B$. Recall that by the definition of support function, we have for each $v\in \sn$,
\begin{equation*}
	\nabla h_K(v)+h_K(v)v=\nu_K^{-1}(v)\in \partial K.
\end{equation*}
Therefore the upper and lower bound on $|\nabla h_K|^2+h_K^2$ follows from the bounds on $K$.
\item Our strategy to prove this statement  is to show that
\begin{enumerate}
	\item The trace of the matrix $(\nabla^2 h_K + h_K I )$, or the sum of the matrix's eigenvalues, is bounded from above. 
	\item The determinant of the matrix $(\nabla^2 h_K + h_K I )$, or the product of the matrix's eigenvalues is bounded both from above and from below (by a positive constant).
\end{enumerate}
Note that (a) and (b), when combined together, immediately imply that all eigenvalues of $(\nabla^2 h_K + h_K I )$ have positive upper and lower bounds---a consequence of the fact that $(\nabla^2 h_K + h_K I )$ is positive definite.

Claim (b) permits a quick proof based on equation \eqref{equaz1}. Indeed, 
\begin{equation*}
	\det (\nabla^2 h_K+ h_KI)=(\sqrt{2\pi})^nfe^{\frac{|\nabla h_K|^2+h_K^2}{2}},
\end{equation*}
where the right side has positive upper and lower bounds based on the bounds of $f$ and statement (1).

To prove Claim (a), let us denote
\begin{equation*}
	H=\text{trace}(\nabla^2 h_K+h_KI)=\Delta h_K + (n-1)h_K.
\end{equation*}
Since $H$ is continuous on $\sn$, there exists $v_0\in \sn$ such that $H(v_0)= \max_{v\in \sn} H$. Then, at $v_0$, we have $\nabla H=0$ and the matrix $\nabla^2 H$ is negative semi-definite. We choose a local orthonormal frame $e_1, \dots, e_{{n-1}}$, such that the Hessian of $h_K$, $(h_{K})_{ij}$, is diagonal. Recall the commutator identity \cite[p. 1361]{MR2237290}:
\begin{equation}
\label{eq local 9030}
	H_{ii}=\Delta w_{ii}-(n-1)w_{ii}+H.
\end{equation} 
We use  $w^{ij}$ to denote the inverse of the matrix $w_{ij}=((h_K)_{ij}+h_K\delta_{ij})$. Equation \eqref{eq local 9030}, the fact that $(w^{ij})$ is positive definite and $\nabla^2 H$ is negative semi-definite, and that $(w^{ij})$ is diagonal, imply that at $v_0$, we have
\begin{equation}
\label{eq local 9001}
	0\geq w^{ii}H_{ii}= w^{ii}\Delta w_{ii} +H\sum\limits_i
w^{ii}-(n-1)^2\geq w^{ii}\Delta w_{ii} -(n-1)^2.
\end{equation}

Taking the logarithm of \eqref{equaz1}, we have 
\begin{equation*}
	\log \det(\nabla^2 h_K+ h_KI) =\log f + \frac{|\nabla h_K|^2+h_K^2}{2}.
\end{equation*}
We take the spherical Laplacian of the above equation and get
\begin{equation}
\label{eq local 9020}
\begin{aligned}
		&\sum_{\alpha}(w^{ij})_\alpha(w_{ij})_\alpha+ w^{ij}\Delta w_{ij} \\
		= &\Delta \log f+\sum_{i,j} (h_K)_{ij}^2+ \sum_{i}(h_K)_i(\Delta (h_K)_i)+ |\nabla h_K|^2+h_K\Delta h_K 
\end{aligned}
\end{equation}
By definition of $H$, we have
\begin{equation}
\label{eq local 9021}
	\sum_{i,j} (h_K)_{ij}^2=H^2-2h_KH+(n-1)h_K^2.
\end{equation}
By definition of $H$ and equation (4.11) in Cheng-Yau \cite{MR0423267}, we have
\begin{equation}
\label{eq local 9022}
	\sum_{i}(h_K)_i(\Delta (h_K)_i)= \sum_i (h_K)_i (\Delta h_K)_i=\nabla h_K\cdot \nabla H-(n-1)|\nabla h_K|^2.
\end{equation}
We also note that 
\begin{equation}
\label{eq local 9023}
	h_K\Delta h_K = h_KH - (n-1)h_K^2.
\end{equation}
Finally, using the fact that $(w^{ij})$ is the inverse matrix of $(w_{ij})$, we get 
\begin{equation*}
	(w^{ij})_\alpha (w_{jk})_\alpha = -w^{im}(w_{ml})_\alpha w^{lj}(w_{jk})_\alpha,
\end{equation*}
which implies its trace is non-positive; that is
\begin{equation}
\label{eq local 9024}
	(w^{ij})_\alpha(w_{ij})_\alpha\leq 0.
\end{equation}
Combining \eqref{eq local 9020}, \eqref{eq local 9021}, \eqref{eq local 9022}, \eqref{eq local 9023}, and \eqref{eq local 9024}, we get
\begin{equation*}
	w^{ij}\Delta w_{ij}\geq \Delta \log f +H^2-h_KH+\nabla h_K\cdot \nabla H-(n-2)|\nabla h_K|^2.
\end{equation*}
When evaluated at $v_0$ where $H$ reaches its maximum, we have
\begin{equation}
\label{eq local 9025}
	w^{ii}\Delta w_{ii}\geq H^2 -h_KH+ (\Delta\log f -(n-2)|\nabla h_K|^2).
\end{equation}
Equations \eqref{eq local 9001} and \eqref{eq local 9025} now imply that
\begin{equation*}
	0\geq H^2-h_KH+(\Delta\log f -(n-2)|\nabla h_K|^2-(n-1)^2).
\end{equation*}
The right side of the above inequality is a quadratic polynomial in $H$. Note that by the bounds on $f$ and $|f|_{C^{2,\alpha}}$, statements (i) and (ii), the coefficients have bounds that only depend on $C$. Therefore, $H$ is bounded from above by a positive constant that only depends on $C$. 

\item By statement (ii), the Monge-Amp\`{e}re equation \eqref{equaz1} is uniformly elliptic. Thus, the standard Evans-Krylov-Safonov theory \cite{MR1814364} implies the higher estimates in statement (iii).
\end{enumerate} 
\end{proof}

\subsection{Existence of smooth solutions via degree theory}

\begin{theorem}[Existence of smooth solutions]
\label{thm existence of smooth solutions}
	Let $0<\alpha<1$ and $f\in C^{2,\alpha}(\sn)$ be a positive even function with $|f|_{L_1}<\frac{1}{\sqrt{2\pi}}$. Then there exists a unique $C^{4,\alpha}$ $o$-symmetric $K$ with $\gamma_n(K)>1/2$ such that
	\begin{equation}
	\label{eq local 7030}
		\frac{1}{(\sqrt{2\pi})^n}e^{-\frac{|\nabla h_K|^2+h_K^2}{2}}\det (\nabla^2 h_K +h_KI) = f.
	\end{equation}
\end{theorem}
\begin{proof}
	The uniqueness part follows from Theorem \ref{thm uniqueness}.

	We use the degree theory for second-order nonlinear elliptic operators developed in Li \cite{MR1026774} for the existence part.
	
	It follows by simple application of intermediate value theorem and Theorem \ref{thm uniqueness} that \eqref{eq local 7030} admits a unique constant solution $h_K\equiv r_0>0$ such that $\gamma_n(K)>1/2$ if $f\equiv c_0>0$ is small enough. We also require that $c_0>0$ is small enough so that $|c_0|_{L_1}<\frac{1}{\sqrt{2\pi}}$. Our final requirement for $c_0>0$ is that the operator $L\phi = \Delta_{\sn} \phi+((n-1)-r_0^{2})\phi$ is invertible. This is possible since spherical Laplacian has a discrete spectrum.
	
	Let $F(\cdot;t):C^{4,\alpha}(\sn)\rightarrow C^{2,\alpha}(\sn)$ be defined as
	\begin{equation*}
		F(h;t) = \det(\nabla^2h+hI)-(\sqrt{2\pi})^ne^{\frac{|\nabla h|^2+h^2}{2}}f_t,
	\end{equation*}
	for $t\in [0,1]$. Here
	\begin{equation*}
		f_t = (1-t)c_0+tf.
	\end{equation*}
	
	Since $f\in C^{2,\alpha}(\sn)$ and $f>0$, there exists $C>0$ such that $\frac{1}{C}<f, c_0<C$ and $|f|_{C^{2,\alpha}}<C$. We let $C'>0$ be the constant extracted from Lemma \ref{lemma apriori estimate}. Note that for each $t\in [0,1]$, the function $f_t$ has the same bound as $f$; namely, $\frac{1}{C}<f_t<C$, $|f_t|_{L_1}<\frac{1}{\sqrt{2\pi}}$, and $|f_t|_{C^{2,\alpha}}<C$.
	
	Define $O\subset C^{4,\alpha}(\sn)$ by
	\begin{equation*}
		O = \left\{h\in C^{4,\alpha}(\sn): \frac{1}{C'}<h<C', \frac{1}{C'}I<(\nabla^2h + hI)<C'I, |h|_{C^{4,\alpha}}<C', \gamma_n(h)>\frac{1}{2}\right\}.
	\end{equation*}
	Here $\gamma_n(h) = \gamma_n([h])$ is well-defined and moreover since $h$ is strictly convex, $h$ is precisely a support function. We note that it is simple to see that $O$ is an open bounded set under the norm $|\cdot|_{C^{4,\alpha}}$.
	
	We also note that since every $h\in O$ has uniform upper and lower bounds for the eigenvalues of its Hessian, the operator $F(\cdot;t)$ is uniformly elliptic on $O$ for any $t\in [0,1]$.
	
	We claim that for each $t\in [0,1]$, if $h\in \partial O$, then
	\begin{equation*}
		F(h; t)\neq 0.
	\end{equation*}
	
	Indeed, if $F(h;t)=0$, then $h$ solves
	\begin{equation}
	\label{eq local 6001}
		\frac{1}{(\sqrt{2\pi})^n}e^{-\frac{|Dh|^2}{2}}\det (\nabla^2h+hI) = f_t.
	\end{equation}
	Since $h\in \partial O$, we also have that $\gamma_n(h)\geq \frac{1}{2}$. If $\gamma_n(h)>1/2$, by Lemma \ref{lemma apriori estimate},  we would have
	\begin{equation*}
		\frac{1}{C'}<h<C', \frac{1}{C'}I<(\nabla^2h + hI)<C'I, |h|_{C^{4,\alpha}}<C'.
	\end{equation*}
	Thus, by the definition of $O$, the only way for $h\in \partial O$ is that
	\begin{equation*}
		\gamma_n(h)=\frac{1}{2}.
	\end{equation*}
	By Corollary \ref{coro local 1001}, we have
	\begin{equation*}
		|S_{\gamma_n, [h]}|\geq \frac{1}{\sqrt{2\pi}}.
	\end{equation*}
	But this contradicts with the fact that $h$ solves \eqref{eq local 6001} and that $|f_t|_1<\frac{1}{\sqrt{2\pi}}$.
	
	Using Proposition 2.2 in Li \cite{MR1026774}, we conclude that
	\begin{equation*}
		\deg (F(\cdot; 0), O, 0) = \deg(F(\cdot; 1), O, 0).
	\end{equation*}
	
	If we can show that $\deg (F(\cdot; 0), O, 0)\neq 0$, then it follows immediately that $$\deg(F(\cdot; 1), O, 0)\neq 0,$$ which then implies the existence of $h\in O$ such that $F(h;1)=0$.
	
	The rest of the proof focuses on showing $\deg (F(\cdot; 0), O, 0)\neq 0$. For simplicity, we will simply write $F(\cdot)=F(\cdot;0)$.
	
	Recall that $r_0>0$ is so that $h_K\equiv r_0$ is the unique solution in $O$ to \eqref{eq local 7030} when $f\equiv c_0$. We denote by $L_{r_0}:C^{4,\alpha}(\sn)\rightarrow C^{2,\alpha}(\sn)$ the linearized operator of $F$ at the constant function $r_0$. It is simple to compute that
	\begin{equation*}
	\begin{aligned}
		L_{r_0}(\phi)
		&= r_0^{n-2}\Delta_{\sn} \phi + ((n-1)r_0^{n-2}-r_0^n)\phi\\
		& = r_0^{n-2}\left(\Delta_{\sn} \phi+((n-1)-r_0^{2})\phi\right).
	\end{aligned}
	\end{equation*}
	Recall that we have specifically chosen a $c_0>0$ so that $L_{r_0}$ is invertible. By Proposition 2.3 in Li \cite{MR1026774} and the fact that $h\equiv r_0$ is the unique solution for $F(h)=0$ in $O$ , we have
	\begin{equation*}
		\deg (F, O, 0) = \deg (L_{r_0}, O,0)\neq 0,
	\end{equation*}
	where the last inequality follows from Proposition 2.4\footnote{Proposition 2.4 in Li \cite{MR1026774} contain some typos, which were corrected by Li on his personal webpage.} in Li \cite{MR1026774}.		
\end{proof}

\subsection{Existence of general $o$-symmetric solutions}
\label{section weak solution}
In this section, we attempt to solve the general measure case by using an approximation argument. The key step is to establish uniform $C^0$ estimate.

\begin{theorem}
\label{thm weak solution}
		Let $\mu$ be an even measure on $\sn$ that is not concentrated in any subspace and $|\mu|< \frac{1}{\sqrt{2\pi}}$. Then there exists a unique origin-symmetric $K$ with $\gamma_n(K)> 1/2$ such that
		\begin{equation*}
			S_{\gamma_n,K} = \mu.
		\end{equation*}
	\end{theorem}
\begin{proof}
	We approximate $\mu$ weakly by a sequence of measures $\mu_i= f_idv$ where $f_i\in C^{2,\alpha}$, $f_i> 0$ with $0<|f_i|_{L_1}<\frac{1}{\sqrt{2\pi}}$. By Theorem \ref{thm existence of smooth solutions}, there are $C^{4,\alpha}$ $o$-symmetric $K_i$ with $\gamma_n(K_i)>\frac{1}{2}$ such that
\begin{equation*}
	\frac{1}{(\sqrt{2\pi})^n}e^{-\frac{|\nabla h_{K_i}|^2+h_{K_i}^2}{2}}\det (\nabla^2 h_{K_i} +h_{K_i}I) = f_i.
\end{equation*}
Since $f_i$ converges weakly to $\mu$, by discarding the first finitely many terms, we may assume that $|f_i|_{L_1}>\varepsilon_0$ for some positive absolute constant $\varepsilon_0$.

We argue that $K_i$ is uniformly bounded. Otherwise, by taking a subsequence, we may assume that $|h_{K_i}|_{C^\infty}\rightarrow \infty$. Now, Lemma \ref{lemma key} tells us that for every $\delta>0$, there exists $v\in \sn$ and $N>0$ such that 
\begin{equation*}
	|\mu_i| <\mu_i(\overline{\xi_{v,\delta}})+\delta,
\end{equation*}
for each $i>N$.
Let $i\rightarrow \infty$. Since $\mu_i$ converges to $\mu$ weakly, we have
\begin{equation*}
	|\mu|\leq \mu(\overline{\xi_{v,\delta}})+\delta.
\end{equation*}
Notice that the above inequality holds for all $\delta>0$. Therefore, 
\begin{equation*}
	|\mu|\leq \mu (v^\perp\cap \sn),
\end{equation*}
which is a contradiction to the fact that $\mu$ is not concentrated in any great subsphere.
\end{proof}

It is of great interest to characterize Gaussian surface area measures for convex bodies that do not necessarily have large Gaussian volume. Example \ref{example 1} in Appendix shows that some essential condition must be found and that how complicated this can be even in the most simple rectangular case on the plane. However, one should keep in mind that the task, even in Example \ref{example 1}, is not to find the relation between the weights of the measure, but rather to give a characterization of the permissible measures. 
	
\section{Appendix}

This appendix consists of two examples. 

The first one shows that the Gaussian Minkowski problem even when restricted to the $o$-symmetric case contains some complications that were masked by our assumption that $\gamma_n(K)>\frac{1}{2}$.

\begin{example}
\label{example 1}
	Consider the even discrete measure on $S^{1}$:
	\begin{equation*}
		\mu = \mu_1 \delta_{e_1}+ \mu_2\delta_{e_2} + \mu_1\delta_{-e_1}+\mu_2 \delta_{-e_2}.
	\end{equation*}
	In order to have an $o$-symmetric convex body $K$ in $\mathbb{R}^2$---in this case, a centered rectangle---so that $\mu=S_{\gamma_2, K}$, the weights $\mu_1$ and $\mu_2$ cannot be chosen independently of each other.
\end{example}
\begin{proof}
	Note that potential solutions for $\mu=S_{\gamma_2, K}$ consists of $o$-symmetric rectangles with sides parallel to the coordinate axes. Let us assume that $K$ is such a rectangle generated by its vertices $(\pm a_1, \pm a_2)$ where $a_1, a_2>0$.  

	It is simple to see that  it has to be the case that $0<\mu_1,\mu_2<\frac{1}{\sqrt{2\pi}}$ since, for example, 
	\begin{equation*}
		\mu_1 = \frac{1}{2\pi}\int_{-a_2}^{a_2} e^{-\frac{a_1^2+y^2}{2}}dy<\frac{1}{2\pi} \int_{-\infty}^\infty e^{-\frac{y^2}{2}}dy=\frac{1}{\sqrt{2\pi}}.
	\end{equation*} 
	
	The claimed codependence between $\mu_1$ and $\mu_2$ is best observed when one of the weights is close enough to $\frac{1}{\sqrt{2\pi}}$. For this purpose, let $\varepsilon_0>0$ be small enough and $\mu_1 = \frac{1}{\sqrt{2\pi}}-\varepsilon_0$. 
	
	It follows from basic computation that 
	\begin{equation}
	\label{eq local 0000}
	\begin{aligned}
		\mu_1 = \frac{1}{\pi} e^{-\frac{a_1^2}{2}}\int_{0}^{a_2} e^{-\frac{t^2}{2}}dt\\
		\mu_2 = \frac{1}{\pi} e^{-\frac{a_2^2}{2}}\int_{0}^{a_1} e^{-\frac{t^2}{2}}dt.
	\end{aligned}
	\end{equation}
	According to \eqref{eq local 0000} and the value of $\mu_1$, we have
	\begin{equation*}
		\frac{1}{\sqrt{2\pi}}-\varepsilon_0<\frac{1}{\pi} e^{-\frac{a_1^2}{2}}\int_{0}^{\infty}e^{-\frac{t^2}{2}}dt = \frac{1}{\sqrt{2\pi}}e^{-\frac{a_1^2}{2}}.
	\end{equation*}
	This implies the existence of $\delta_0=\sqrt{-2\ln (1-\sqrt{2\pi}\varepsilon_0)}>0$ such that $0<a_1<\delta_0$.
	
	Now, using the second equation in \eqref{eq local 0000}, we have
	\begin{equation*}
		\mu_2 <\frac{1}{\pi} \int_{0}^{\delta_0} e^{-\frac{t^2}{2}}dt.
	\end{equation*}
	
	This crude analysis shows that when one of the weights is close enough to its maximally allowable value, then the other weight must be correspondingly close enough to $0$. Moreover, the dependence takes a complicated nonlinear form involving Gaussian distribution function. A similar, but arguably much more complicated computation can show that this phenomenon happens to any $o$-symmetric polygons on the plane.
	
	Note that in the above computation, $|\mu|>2(\frac{1}{\sqrt{2\pi}}-\varepsilon_0)$, which is excluded by the hypotheses in Theorem \ref{thm weak solution}.
\end{proof} 

It follows from a trivial calculation that the Gaussian surface area measure $S_{\gamma_n, rB}$ of a centered ball of radius $r$ is given by
\begin{equation*}
	dS_{\gamma_n, rB}(v) = \frac{1}{(\sqrt{2\pi})^n} e^{-\frac{r^2}{2}}r^{n-1}dv. 
\end{equation*}
Because of the behavior of $e^{-\frac{r^2}{2}}r^{n-1}$ on $(0,\infty)$, it is simple to conclude that for every $c>0$ that is small enough, there are exactly two balls $r_1B$ and $r_2B$ whose Gaussian surface area measure is given by $cdv$. When this is combined with our uniqueness result (Theorem \ref{thm uniqueness}), it is tempting to think that uniqueness also holds when restricted to convex bodies whose Gaussian volume is sufficiently small. This turned out to be false, as illustrated by the following example.
\begin{example}
	As in Example \ref{example 1}, consider an $o$-symmetric rectangle $K$ in $\mathbb{R}^2$ generated by its vertices $(\pm a_1, \pm a_2)$. Suppose its Gaussian surface area measure is given by 
	\begin{equation*}
		dS_{\gamma_n, K} = \mu_1 \delta_{e_1}+ \mu_2 \delta_{e_2}+\mu_1\delta_{-e_1} + \mu_2 \delta_{-e_2}.
	\end{equation*}
	As has already been discussed in Example \ref{example 1}, when $\varepsilon_0>0$ is sufficiently small, the weight $\mu_1  =\frac{1}{\sqrt{2\pi}}-\varepsilon_0$ is an allowable choice. In particular, we just need to choose sufficiently small $0<a_1<\delta_0$ and correspondingly a big enough $a_2$ so that 
	\begin{equation}
	\label{eq local 0001}
		\frac{1}{\sqrt{2\pi}}-\varepsilon_0=\mu_1 = \frac{1}{\pi} e^{-\frac{a_1^2}{2}}\int_{0}^{a_2} e^{-\frac{t^2}{2}}dt\
	\end{equation}
	holds.
	
	It is simple to note that by
	\begin{equation*}
		\mu_2 = \frac{1}{\pi} e^{-\frac{a_2^2}{2}}\int_{0}^{a_1} e^{-\frac{t^2}{2}}dt,
	\end{equation*}
	when  $a_1 \rightarrow 0$, we have $\mu_2\rightarrow 0$. 
	
	Note also that by \eqref{eq local 0001}, when $a_1\rightarrow \delta_0$, we have $a_2\rightarrow \infty$. Therefore $\mu_2\rightarrow 0$ as well in this scenario. 
	
	By intermediate value theorem, we can conclude that when $\mu_1=\frac{1}{\sqrt{2\pi}}-\varepsilon_0$ and $\mu_2$ is sufficiently small, there are two rectangles---one whose $a_1$ is close to $0$, the other whose $a_1$ is close to $\delta_0$---such that they have the same Gaussian surface area measure.
\end{example}

\bibliography{mybib}{}
\bibliographystyle{plain}
\end{document}